\documentclass[11pt, reqno]{amsart}
\usepackage{amssymb}
\usepackage{latexsym}
\usepackage{amsmath}
\usepackage{amsfonts}
\usepackage[hidelinks]{hyperref}
\hypersetup{
  colorlinks   = true, 
  urlcolor     = blue, 
  linkcolor    = blue, 
  citecolor   = blue 
}
\usepackage{enumitem}
\usepackage{fancyhdr}
\usepackage{cleveref}
\usepackage{mathrsfs}
\usepackage[dvipsnames]{xcolor}
\usepackage{soul}

\usepackage{color}

\theoremstyle{plain}
\newtheorem{theorem}{Theorem}[section]

\newcommand{\R}{\ensuremath{\mathbb{R_+}}}

\newtheorem{lemma}[theorem]{Lemma}

\newtheoremstyle{remark}
    {} 
    {} 
    {}          
    {}          
    {\bfseries} 
    {.}         
    {.5em}      
    {}          

\theoremstyle{remark}
\newtheorem{remark}{Remark}[section]

\newcommand{\red}[1]{\textcolor{red}{#1}}

\newtheoremstyle{example}
    {\dimexpr\topsep/2\relax} 
    {\dimexpr\topsep/2\relax} 
    {}          
    {}          
    {\bfseries} 
    {.}         
    {.5em}      
    {}          

\theoremstyle{example}

\newtheoremstyle{definition}
    {\dimexpr\topsep/2\relax} 
    {\dimexpr\topsep/2\relax} 
    {}          
    {}          
    {\bfseries} 
    {.}         
    {.5em}      
    {}          

\theoremstyle{definition}
\newtheorem{definition}{Definition}[section]

\newtheorem*{similartheorem*}{Theorem \dualnumber{$'$}}

\numberwithin{equation}{section}

\newcommand{\Ha}{\mathbb{H}^n}

\def\R{\mathbb{R}}
\def\H{\mathbb{H}}
\def\C{\mathbb{C}}
\def\Z{\mathbb{Z}}

\allowdisplaybreaks

\begin{document}
\title[Lacunary maximal functions on homogeneous groups]{Weighted estimates for lacunary maximal functions on homogeneous groups}

\keywords{Homogeneous group,
Heisenberg group, Lacunary maximal function, Weighted $L_p$ spaces}
{\let\thefootnote\relax\footnote{\noindent 2010 {\it Mathematics Subject Classification.} Primary 42B20, 26A33; Secondary 43A85}}

\author{Abhishek Ghosh and Rajesh K. Singh}
\address[Abhishek Ghosh] {Institute of Mathematics, Polish Academy of Sciences\\
\'{S}niadeckich 8\\
00-656 Warsaw\\
Poland}
\email{abhi170791@gmail.com}

\address[Rajesh K. Singh]{Department of Mathematics, Gitam Institute of Science, GITAM University, Visakhapatnam- 530045, A.P., India}
\email{rsingh4@gitam.edu}

\thanks{This research was funded in part by the National Science Center, Poland, grant 2021/43/D/ST1/00667.}

\pagestyle{headings}

\begin{abstract}
In this article, we study weighted estimates for a general class of lacunary maximal functions on homogeneous groups. As an application we derive improved weighted estimates for the lacunary maximal function associated to the Kor\'anyi spherical means as well as for the lacunary maximal function associated to codimension two spheres in the Heisenberg group.
\end{abstract}

\maketitle

\section{Introduction and Preliminaries}
Spherical maximal functions are widely studied object in Hramonic analysis on Euclidean spaces and in homogeneous groups. Given a Schwartz function $f,$ the spherical maximal function $\mathcal S$ is defined  by
\[
\mathcal{S}f(x)=\sup_{t>0}\biggl|\int_{{\mathbb{S}}^{n-1}}f(x-ty)\, d \sigma_{n-1}(y)\biggr|, \quad x \in \mathbb R^n,
\]
where  $d \sigma_{n-1}$   is the rotationally invariant measure on $\mathbb{S}^{n-1}$, normalized such that $ \sigma_{n-1}(\mathbb{S}^{n-1})=1$. While we would not delve into a detailed historical account here; in a pioneering work in \cite{SteinSph}, Stein proved that $\mathcal{S}$ is bounded on $L^p(\mathbb{R}^n)$, if and only if $n\geq 3$ and $p>n/(n-1).$ He also showed examples proving that the boundedness fails if $p\leq \frac{n}{n-1}$ and $n\geq 2$. Later, in \cite{Bourgain} Bourgain proved it for dimension two, concluding that $\mathcal{S}$ is bounded on $L^p(\mathbb{R}^2)$ if $p>2.$ In this article, our primary concern is to study weighted inequalities for lacunary maximal functions on homogeneous groups with a particular emphasis on the Heisenberg group. To state our main results, let us recall preliminaries regarding homogeneous groups, we mainly follow \cite{Folland-Stein} and \cite{HickmanJFA}.

A homogeneous group $G$ is a connected, simply connected nilpotent Lie group with associated real finite-dimensional Lie algebra ${\mathfrak g}$ such that ${\mathfrak g}$ is endowed with a family of algebra automorphisms $\{\delta_r\}_{r>0}$. Moreover, each $\delta_r$ is of the form $\delta_r= \exp(\Gamma\log r),\, \, r>0,$ where $\Gamma$ is a diagonalisable linear operator on ${\mathfrak g}$ with positive eigen-values. Since, in this case, the exponential map is a global diffeomorphism, $G$ can be identified with $\mathbb{R}^n$ and the Haar measure on $G$ can be identified with the $n$-dimensional Lebesgue measure. Therefore, the topological dimension of $G$ is $n$ and homogeneous dimension of $G$ is the quantity $Q=:\text{trace}(\Gamma)=\sum_{j=1}^{n}d_j,$ where $d_{j}>0$ are eigen-values of $\Gamma$ counted with multiplicity. Also, the maps $\exp  \circ \, \delta_{r}\circ \exp^{-1}$ are group automorphisms on $G$ and will be denoted by $\delta_{r}$ again for simplicity and will be called dilations of $G$. Finally, let $|\cdot|$ be a homogeneous quasi-norm, i.e.,  $|x| = 0$ if and only if $x=0$ where $0$ denotes the group identity, and $|\delta_r x | = r |x|$ for all $r>0$ and $x\in G,$ and there exists a constant $c_{G}\geq 1$ such that $|x.\, y|\leq c_{G}(|x|+|y|)$ for all $x, y\in G$. Having defined the preliminaries, we now focus on our main objects of study. Given a finite, compactly supported Borel measure $\mu$ on $G$ and $r>0$, define the dilates $\mu_r$ of $\mu$ as
$$\langle \mu_r,\phi \rangle := \int_{G} \phi( \delta_{r} (x))\,d\mu(x) \quad$$
for $\phi \in C_c(G)$, the space of compactly supported continuous functions on $G$. For any $k \in \Z $, we denote by $\mu_{k}$ the dyadic dilate of $\mu$, that is, $\mu_{k} := \mu_{2^{k}}$. Also, the reflection is defined as $$\langle{\tilde{\mu}}, {\phi}\rangle:= \int_{G} \phi(x^{-1})\,d\mu(x).$$ 

Following \cite{HickmanJFA}, we define maximal functions on the homogeneous groups. For $f \in C_{c}(G)$, we consider the averaging operators associated to such measures $\mu$ as following 
\begin{equation*}
    S[\mu]f(x) := f \ast \mu(x) = \int_{G} f(x \cdot y^{-1}) \,d\mu(y),  \qquad \textrm{ for $x \in G$.}
\end{equation*}
Associated to the measure $\mu$, the lacunary maximal function acting on $f \in C_{c}(G)$, is defined as 
\[
\mathcal{S}_{lac}[\mu]f(x) := \sup_{ k\in \Z}\big|S[\mu_k]f(x)\big|, \qquad  x \in G,
\]  
and the full maximal function as
\begin{equation*}
\mathcal{S}_{full}[\mu]f(x):= \sup_{r>0}\big|S[\mu_r]f(x)\big|\qquad \textrm{for}\, x \in G.
\end{equation*}
The study of lacunary maximal functions goes back to the work of C. P. Calderon in \cite{C}, where he proved that the lacunary spherical maximal function
$$\mathcal{S}_{lac}f(x)=\sup_{k\in \mathbb{Z}}\biggl|\int_{\mathbb{S}^{n-1}}f(x-2^{k}y)\, d\sigma_{n-1}(y)\biggr|, \quad x \in \mathbb{R}^n,$$
is bounded on $L^p(\mathbb{R}^n)$ for all $1<p<\infty.$ Recently, this result is extended in vast generality on homogeneous groups in \cite{HickmanJFA} and we are interested in proving weighted estimates for the lacunary maximal functions $\mathcal{S}_{lac}[\mu]$ on homogeneous group $G.$ We describe statements of our main results in the following section.
\subsection{Statement of main results}
The point of departure for our article stems from the very recent work of Sheri--Hickman--Wright in \cite{HickmanJFA} where the authors proved $L^p$ boundedness of general lacunary maximal functions
$\mathcal{S}_{lac}[\mu]$ on homogeneous groups $G$ under the assumption that $\mu$ satisfies the \textit{Curvature assumption (CA),} which says that there exists  natural number $N$ such that the iterated convolution $\mu^{(N)}$ is absolutely continuous with respect to the Haar measure of $G$ and there exists $\gamma>0, C_{\mu}>0$ such that the Radon-Nikodym derivative of $\mu^{(N)}$, say $h,$ satisfies 
\begin{align}
\int_{G}\big(|h(xy^{-1})-h(x)|+|h(y^{-1} x)-h(x)|\big)\, dx\leq C_{\mu}\, |y|^{\gamma}   
\end{align}
for all $y\in G.$ The iterated convolutions are defined as $\mu^{(2k+1)}=\mu^{(2k)}\ast \tilde{\mu}$ and $\mu^{(2k)}=\mu^{(2k-1)}\ast \mu$ for $k\in \mathbb{N}.$  The main result in \cite{HickmanJFA} states that
\begin{theorem}[\cite{HickmanJFA}]
\label{Hickman-main}
Let $G$ be a homogeneous group and let $ \mu$ be a finite, compactly supported Borel
measure on $G$ satisfying the \textit{Curavture assumption (CA)}. Then $\mathcal{S}_{lac}[\mu]$ maps $L^p(G)$ to $L^p(G)$ for $1<p\leq \infty.$
\end{theorem}
In our first main result we address weighted estimates for the lacunary maximal function $\mathcal{S}_{lac}[\mu]$ on homogeneous groups $G.$  We need a brief description of Muckenhoupt weights on homogeneous groups. Since we are defining $f\ast \mu$ as our spherical mean, it is necessary to work with
the left invariant metric $d_{L}(x, y)= |x^{-1}\cdot y|$ on $G$. Thus, for $a \in G$, the ball of radius $r$ centered at $a$ is given by $B(a, r):= a \cdot B(0,r)$. We now recall the definition of $\mathcal{A}_p (G)$ weights under this quasimetric.  
\begin{definition}
Let  $1<p<\infty,$ a weight $\omega$ in $\mathcal{A}_{p}(G)$ is a locally integrable function on $G$ such that $w>0$ a.e. and 
\[\sup_{B}\bigg(\frac{1}{|B|}\int_{B} \omega\,  dx\bigg) \bigg(\frac{1}{|B|}\int_{B} \omega^{-\frac{1}{p-1}}\,  dx\bigg)^{p-1}<\infty,\]
where the supremum is taken over all balls $B$ in $G.$
\end{definition}
The Hardy-Littlewood maximal function on $G,$ will be denoted by $\mathfrak{M}$ and is defined as $$\mathfrak{M}f(x)=\sup\limits_{B\ni x} \frac{1}{|B|}\int_{B}|f(y)|\, dy,$$
where the supremum is taken over all balls $B$ containing the point $x$ and $f$ is a locally integrable function on $G$. We simply denote the weighted spaces $L^p(G, \omega)$ by $L^p(\omega)$ unless there is a confusion. It is well known that for $1<p<\infty,$ $\mathfrak{M}: L^p(G, \omega)\to L^p(G, \omega)$ if and only if $\omega\in \mathcal{A}_{p}(G).$ We are at a position to state our main results. Our first result concludes weighted estimates for the lacunary maximal function $\mathcal{S}_{lac}[\mu]$ once we have uniform weighted estimates for the single averaging operators $S[\mu_k].$ 
\begin{theorem}
\label{main:suff}
Let $G$ be a homogeneous group and let $\mu$ be a finite, compactly supported Borel
measure on $G$ satisfying the \textit{Curvature assumption (CA)}. Let $1<p<\infty$ and $\omega\in \mathcal{A}_{p}(G).$ If there exists a constant $C,$ independent of $k,$ such that 
\begin{align}
\label{uniform}
\int_{G}|S[\mu_k]f(x)|^p\, \omega(x)\, dx\leq C \int_{G} |f(x)|^p\, \omega(x)\, dx     
\end{align}
holds for all $k\in\mathbb Z$ then $\mathcal{S}_{lac}[\mu]: L^p(\omega^s) \to L^p(\omega^s)$  for all $0\leq s<1.$
\end{theorem}


For the special case of power weights, the following theorem concerns the end-point cases which are not covered by the previous theorem. 
\begin{theorem}
\label{Lac-end}
Let $G$ be a homogeneous group and let $\mu$ be a finite, compactly supported Borel
measure on $G$, supported away from the origin and satisfying the \textit{Curvature assumption (CA)}. Let $\alpha<0,$ and $1<p<\infty.$ If there exist a constant $C$, independent of $k \in \Z$, such that 
\begin{align}
\label{unoform-end}
\int_{G}|S[\mu_k]f(x)|^p\, |x|^{\alpha}\, dx\leq C \int_{G} |f(x)|^p\, |x|^{\alpha}\, dx.     
\end{align}
Then,
\begin{align}
\label{lac-end}
\int_{G}|S_{lac}[\mu]f(x)|^p\, |x|^{\alpha}\, dx\leq C \int_{G} |f(x)|^p\, |x|^{\alpha}\, dx. 
\end{align}
\end{theorem}

As an application of Theorem~\ref{main:suff} and Theorem~\ref{Lac-end}, we provide the following  results concerning the power weights for the lacunary maximal functions on the Heisenberg group,  $\Ha.$ We begin with a brief description of the Heisenberg group. The Heisenberg group, $\mathbb{H}^n$, is the two step nilpotent Lie group with underlying manifold $\mathbb{C}^n \times \mathbb{R}$ associated with the group law
\begin{equation}
(z,t) \cdot (w,s) := \left(  z + w, t + s + \frac{1}{2}  \Im(  z . \bar{w}) \right),\ \ \text{for all} \ \ (z,t), (w,s) \in \mathbb{H}^n.
\end{equation} 
We have a family of non-isotropic dilations defined by $\delta_{r}(z,t):=(rz,r^2t)$, for all $(z,t) \in \mathbb{H}^n$, for every $r>0$. The Kor\'anyi norm on $\mathbb{H}^n$ is defined by
 \begin{equation*}
     |(z,t)|:= \left( \|z \|^4 + t^2 \right)^{\frac{1}{4}}, \ \ \    (z,t) \in \mathbb{H}^n, 
 \end{equation*}
which is homogeneous of degree 1, that is $|\delta_{r}(z,t)|= r \,  |(z,t)|$. Here, $\| z \|$ denotes the Euclidean norm of $z \in \C^{n}$. The Haar measure on $\mathbb{H}^n$ coincides with the Lebesgue measure $dz dt$. Let $B(0,r):=\{ (z,t) \in \H^n : |(z,t)| < r \}$ be the ball of radius $r$ with respect to Kor\'anyi norm. One has its measure $|B(0,r)| = C_{Q} \, r^{Q},$ where $Q=(2n + 2)$ is known as the homogeneous dimension of $\H^n$. Moreover, $d=2n+1$ will denote the \textit{topological dimension} of the Heisenberg group $\H^n$. The convolution of $f$ with $g$ on $\H^{n}$ is defined by
\begin{equation*}
    f * g \, (x) = \int_{\H^{n}}  f(x y^{-1}) g(y) dy, \ \ \ x \in \H^{n}.
\end{equation*}
Given $f$ on $\Ha,$ $\delta_r f $  denotes $ \delta_r f(z,t)=f(\delta_r(z,t))$, for all $(z,t)\in \Ha.$ Also, $ \sigma_r = \delta_r \sigma $ where $\sigma_r$ is defined by $\langle \sigma_r, \phi \rangle :=\langle \sigma, \delta_r \psi\rangle$ for all $\psi\in C^{\infty}_{c}(\Ha).$ We recall the Kor\'anyi sphere 
\[\Sigma_{K}:=\{(u, s)\in \mathbb{H}^n: |(u, s)|=\left( \|u \|^4 + s^2 \right)^{\frac{1}{4}}=1 \}.\]

Let $\sigma$ denotes the normalised surface measure on $\Sigma_{K}$ induced by the Lebesgue measure on $\mathbb{R}^{2n+1}$. Finally, we define the Kor\'anyi spherical mean $S_r$ by 
\begin{align}
S_r f(x) = f\ast \sigma_r(x) &= \int_{|y|=1} f(x \cdot \delta_r y^{-1}) \, d\sigma(y)\nonumber\\
&=\int_{|(u, s)|=1} f\big(z-ru, t-r^2s- {\textstyle\frac{1}{2} } r \Im(  z . \bar{u})\big) \, d\sigma(u, s) 
\end{align}
For $k\in \mathbb{Z},$ we simply denote
$$S_k f(x)=f\ast \sigma_k(x)= \int_{|y|=1} f(x \cdot \, \delta_{2^k} y^{-1}) \, d\sigma(y).$$
For $f\in C_{c}(\Ha),$ the lacunary maximal functions on $\Ha$ is given by
$$\mathcal{M}_{lac}f(x):=\sup_{k\in \mathbb{Z}} | S_k f(x)|, \ \ x \in \Ha.$$
Our result for the lacunary maximal function $\mathcal{M}_{lac}$ asserts
\begin{theorem}
\label{main:koranyi}
Let $d=(2n+1)$ denotes the topological dimension of $\Ha,$ and $1<p< \infty$. Then, 
$$\int_{\mathbb{H}^n} |\mathcal{M}_{lac}f(x)|^p\, |x|^{\alpha}\, dx\leq C_{n, \alpha} \int_{\mathbb{H}^n} |f(x)|^p\, |x|^{\alpha},$$ 
holds for all $f\in L^p(\mathbb{H}^n, |x|^\alpha)$ provided
$-(d-1)\leq \alpha< (d-1)(p-1).$
\end{theorem}

\begin{remark}
We would like to remark that, compared to previously known results, the above results concludes weighted boundedness for  $\mathcal{M}_{lac}$ for a significantly larger range of $\alpha,$ and we show it details in Section~\ref{comparison}.   
\end{remark}

We shall also  consider averages over horizontal Euclidean spheres in the Heisenberg group. Let $d \sigma_{2n-1}$ denote the normalised surface measure on the Horizontal Sphere $\mathbb{S}^{2n-1} \times \{0\} \subseteq \R^{2n+1}$. The spherical mean over horizontal spheres will be denoted by the following
\begin{equation}\label{eq: horizontal spherical av}
    H_{r} f(z, t) = \int_{\mathbb{S}^{2n-1}} f\big(z-r u, t- {\textstyle\frac{1}{2} } r  \Im(z. \bar{u})\big) \,d \sigma_{2n-1}(u), \qquad (z, t) \in \Ha.
\end{equation}
and the associated lacunary maximal function is 
$${M}_{lac}f(x):=\sup_{k\in \mathbb{Z}} | H_k f(x)|, \ \ x \in \Ha.$$
Here, $H_k$ is the shorthand for $H_{2^k}$. Our result for the lacunary maximal function ${M}_{lac}$ reads as follows.
\begin{theorem}
\label{main:hori}
Let $n \geq 2$ and let $1<p< \infty$. Then,
$$\int_{\mathbb{H}^n} |M_{lac}f(x)|^p |x|^{\alpha}\, dx\leq C_{n, \alpha} \int_{\mathbb{H}^n} |f(x)|^p |x|^{\alpha} dx,$$ 
holds for all $f\in L^p(\mathbb{H}^n, |x|^\alpha)$ provided $-(2n-1)\leq \alpha< (2n-1)(p-1).$
\end{theorem}
The maximal functions $\mathcal{M}_{lac}$ and $M_{lac}$ are extensively studied objects. For an historical account, in \cite{Cowling}, Cowling studied the spherical maximal function over Kor\'anyi spherical averages and he proved that the spherical maximal function $\mathcal{M}_{full}f(x)=\sup_{r>0}|S_{r}f(x)|$ is bounded on $L^p(\Ha) $ for all $p>\frac{2n+1}{2n}=\frac{d}{d-1}.$ Motivated by the work in \cite{Lacey}, in recent times, $\mathcal{M}_{full}$ and $\mathcal{M}_{lac}$ have been thoroughly studied in \cite{PT} and in \cite{Rajula} in the context of weighted inequalities, sparse bounds and $L^p$ improving estimates. We will place our results in comparison with that of \cite{PT, Rajula} in the next subsection for a proper perspective.

Averages over the horizontal spheres were first studied by Nevo and Thangavelu in the seminal article \cite{NeT} and they proved that $M_{full}f(x)=\sup_{r>0}|H_{r}f(x)|$ is bounded on  $L^p(\Ha)$ for $p>\frac{2n-1}{2n-2}.$ Subsequently, it was proved that $M_{full}f(x)$ maps $L^p(\Ha)$ is bounded on  $L^p(\Ha)$ if and only if $p>\frac{2n}{2n-1}$ independently by Narayanan and Thangavelu in \cite{NaT} and in \cite{MS}  by M\"uller and Seeger. Recently, in the article \cite{BHRT} the authors proved that ${M}_{lac}$ is bounded from $L^p(\Ha)$ to itself for $1<p<\infty.$ They also proved very crucial sparse operator bounds and renewed the interest in the study of ${M}_{lac}$ and influenced a lot of research in this direction. These results were further sharpened in \cite{Roos}. The following section will highlight the importance of our results, namely, Theorem~\ref{main:koranyi} and Theorem~\ref{main:hori}, in comparison with existing weighted boundedness results available in \cite{Rajula, Roos}, see also \cite{PT}.

\subsection{Comparison with previously known results}
\label{comparison}
An important class of $\mathcal{A}_{p}(G)$ weights are given by the power weights, i.e., weights of the form $\omega_{\alpha}(x)=|x|^{\alpha}.$ Recall the polar decomposition on the homogeneous group $G.$ Denote $\Sigma:=\{x\in G: |x|=1\},$ then there exists a unique radon measure $\vartheta$ on $\Sigma$ such that
\begin{align}
\label{polar}
\int_{G} f\, dx=\int_{0}^{\infty} \int_{\Sigma} f(\delta_{r} y) r^{Q-1}\, d\vartheta(y)\, dr.    
\end{align}
 Using the polar decomposition \eqref{polar}, one can prove that $\omega_{\alpha}(x)=|x|^{\alpha} \in \mathcal{A}_{p}(G)$ if and only if $-Q<\alpha<Q(p-1),$ where $Q$ is the homogeneous dimension of $G$, defined earlier.

In this section, we describe how our results improve on the existing results in \cite{PT, Rajula} and \cite{Roos} in the context of power weights on the Heisenberg group. We only explain the case of $\mathcal{M}_{lac}$, as other cases can be handled similarly. Recall $d=2n+1$ and $Q=2n+2$
denote the topological and homogeneous dimension of $\Ha$, respectively. For $p\in (1, \infty),$ we say $\omega$ belongs to the Reverse-H\"older class of weights, $RH_{p}(\Ha)$, if
\begin{align}
 \frac{1}{|B|} \int_{B} \omega^p \leq C_{\omega, n}\left(\frac{1}{|B|}\int_{B} \omega\right)^{p},  
\end{align}
holds for all balls $B$ in $G$. The following theorem was proved in \cite{Rajula}.
\begin{theorem}[\cite{Rajula}]
\label{rajula-sparse}
Let $n\geq 1$ and let $\phi_{lac}$ be a real valued function such that $\frac{1}{\phi_{lac}}$ is a piecewise linear function on $[0, 1]$ whose graph connects the points $(0, 1), (\frac{d}{d+1}, \frac{d}{d+1}),$ and $(1, 0).$ Then $\mathcal{M}_{lac}$ is bounded from $L^p(\Ha, \omega)$ to itself if
$$\omega\in \mathcal{A}_{p/r}(\Ha)\cap RH_{\left(\frac{\big(\phi_{lac}(\frac{1}{r})\big)'}{p}\right)'}(\Ha),$$
where $1<r<p<\big(\phi_{lac}(\frac{1}{r})\big)'.$
\end{theorem}
Next, we compute the range of $\alpha$ provided by the above \Cref{rajula-sparse} when applied to power weight $\omega_{\alpha} = |x|^{\alpha}$. Thus, $\alpha$ should be such that it satisfy
\begin{equation}
\label{power-rhi}
\omega_{\alpha} =|x|^\alpha\in \mathcal{A}_{p/r}(\Ha)\cap RH_{\left(\frac{\big(\phi_{lac}(\frac{1}{r})\big)'}{p}\right)'}(\Ha).  
\end{equation}
According to the Theorem~\ref{rajula-sparse}, the function $\frac{1}{\phi_{lac}}$ is the following:
\begin{align*}
\frac{1}{\phi_{lac}(t)}=\begin{cases} 1-\frac{t}{d}, \,\,0<t\leq \frac{d}{d+1},\\
d(1-t),\,\, \frac{d}{d+1}\leq t<1.      
\end{cases}    
\end{align*}
Its H\"older conjugate is given by
\begin{align*}
\frac{1}{(\phi_{lac}(t))'}=\begin{cases}\frac{t}{d}, \,\,0<t\leq \frac{d}{d+1},\\
1-d(1-t),\,\, \frac{d}{d+1}\leq t<1.      
\end{cases}    
\end{align*}
We take $p$ such that $\frac{d+1}{d}<p<\infty,$ that is, $\frac{1}{p}<\frac{d}{d+1}$ and choose $r$ such that $\frac{1}{r}=\frac{d}{d+1}-\epsilon,$ then
$$\frac{1}{(\phi_{lac}(\frac{1}{r}))'}=\frac{d-(d+1)\epsilon}{d(d+1)},\,\,\,\  \frac{1}{r}=\frac{d-(d+1)\epsilon}{d+1}.$$
The intersection of $\mathcal{A}_{p}(\Ha)$ and $RH_{p}(\Ha)$ weights can be seen as
\begin{align}
\label{rhi}
\omega\in \mathcal{A}_{p/r}(\Ha)\cap RH_{\big(\frac{s}{p}\big)'}(\Ha) \iff  \omega^{\big(\frac{s}{p}\big)'}\in \mathcal{A}_{\big(\frac{s}{p}\big)'\big(\frac{p}{r}-1\big)+1}(\Ha).    
\end{align}
See equation (12) in \cite{BC} for the above. Using \eqref{power-rhi} and \eqref{rhi}, we further deduce
\begin{align}
|x|^\alpha\in \mathcal{A}_{p/r}(\Ha) & \cap RH_{\big(\frac{(\phi_{lac}(1/r))'}{p}\big)'} (\Ha) \\
& \iff |x|^{\alpha \frac{d(d+1)}{d(d+1)-p(d-(d+1)\epsilon)}}\in \mathcal{A}_{\frac{d(d+1)}{d(d+1)-p(d-(d+1)\epsilon)}\left(p\frac{d-(d+1)\epsilon}{d+1}-1\right)+1}(\Ha).\label{inter}    
\end{align}
Recall that $|x|^\beta\in \mathcal{A}_p(\Ha)\iff -Q<\beta<Q(p-1)$. Thus, \eqref{inter} holds true if and only if 
\begin{align}
-Q<\alpha \  \frac{d(d+1)}{d(d+1)-p(d-(d+1)\epsilon)}<Q \,  \frac{d(d+1)}{d(d+1)-p(d-(d+1)\epsilon)}\left(p\frac{d-(d+1)\epsilon}{d+1}-1\right).  
\end{align}
Subsequently, we focus on $\alpha\geq 0,$ so the above implies
\begin{align}
\alpha<Q\left(p\frac{d-(d+1)\epsilon}{d+1}-1\right).
\label{pos-alpha}
\end{align}
Further, the condition $1<r<p<\phi_{lac}(1/r)'$ implies
\begin{align}
\frac{d+1}{d-(d+1)\epsilon}<p< \frac{d(d+1)}{d-(d+1)\epsilon}\iff d\left(\frac{1}{d+1}-\frac{1}{p}\right)<\epsilon<\left(\frac{d}{d+1}-\frac{1}{p}\right).
\end{align}
For $\epsilon$ to be positive we must have
$p>(d+1)$ and taking $\epsilon \to \left(d\left(\frac{1}{d+1}-\frac{1}{p}\right)\right)^{-}$ in \eqref{pos-alpha} we get
$$\alpha<Q(d-1).$$
On the other hand, our main result Theorem~\ref{main:koranyi} implies that
$$\int_{\mathbb{H}^n} |\mathcal{M}_{lac}f(x)|^p |x|^{\alpha}\, dx\leq C_{n, \alpha} \int_{\mathbb{H}^n} |f(x)|^p |x|^{\alpha}$$
holds for all $\alpha$ as $-2n\leq \alpha< 2n(p-1).$ Therefore, we obtain a strict improvement over the previously known results in the context of power weights. For other values of $p$ also one can make similar observations following the above line of ideas.

We close this section by a standard weighted estimate for the Littlewood-Paley square function on Homogeneous groups which will be useful for our purpose. Let $\eta$ be a smooth compactly supported function supported in annulus $B(0, 1)\setminus B(0, 1/2)$ which satisfies $\int_{G} \eta=1,$ $\eta(x)=\eta(x^{-1})$ and $\eta\geq 0$ on $G$. Define
\[\phi_{j}(x) : =2^{-(j-1)Q}\eta(\delta_{2^{-(j-1)}}(x))-2^{-jQ}\eta(\delta_{2^{-j}}(x)), \ \ x \in G.\]
The function $\eta$ can be chosen appropriately so that $f=\sum_{j \in \Z} f\ast \phi_{j}$ holds for all $f\in C_{c}(G).$ The following weighted estimate was proved in \cite{Sato2013}.
\begin{lemma}[Lemma~6, \cite{Sato2013}]
\label{square-lemma}
Let $1<p<\infty$ and $\omega\in \mathcal{A}_{p}(G).$ Then,
\begin{align}
\bigg\|\left(\sum_{j \in \Z} \big|f\ast \phi_{j  }\big|^2\right)^{\frac{1}{2}}\bigg\|_{L^p(G,\, \omega)}\leq C_{p, \omega} \big\|f\big\|_{L^p(G,\, \omega)}.
\end{align}
\end{lemma}

In the following section~\ref{sec-suff} we provide the proof of Theorem~\ref{main:suff} and Theorem~\ref{Lac-end}. The subsequent section~\ref{Heisenberg-max} contains several preliminary lemmas and the proofs of Theorem~\ref{main:koranyi} and Theorem~\ref{main:hori}.

\section{Proof of main results}
\label{sec-suff}
This section is dedicated to prove our main results. We first prove the Theorem~\ref{main:suff} and as an application we deduce Theorem~\ref{main:koranyi} and Theorem~\ref{main:hori}. A key ingredient in our proof of Theorem~\ref{main:suff} is the $L^p$ boundedness of the lacunary maximal functions, Theorem~\ref{Hickman-main}, proved by Sheri--Hickman--Wright in \cite{HickmanJFA}.

\begin{proof}[Proof of Theorem~\ref{main:suff}]

The proof relies on a nice application of the Littlewood-Paley theory. Let $\eta$ be a smooth compactly supported function supported in $B(0, 1)\setminus B(0, 1/2)$ and satisfies $\int_{G} \eta=1,$ $\eta(x)=\eta(x^{-1})$ and $\eta\geq 0$ on $G$. Define
\[\phi_{j}(x)=2^{-(j-1)Q}\eta(\delta_{2^{-(j-1)}}(x))-2^{-jQ}\eta(\delta_{2^{-j}}(x)).\]
Moreover, by the choice of $\eta$, we have that $f=\sum_{j \in \Z} f\ast \phi_{j}$ holds for all $f\in C_{c}(G).$ Therefore, using this decomposition of the function $f,$ we can write that

\begin{align}
\label{dyadic-piece}
\mathcal{S}_{lac}[\mu]f(x)\leq \sum_{j\in \mathbb{Z}} \sup_{k\in \mathbb{Z}} |f\ast \phi_{k+j}\ast \mu_{k}|.    
\end{align}

Let $1<p<2,$ also assume that $0\leq s<1$. Moreover, for $L\in \mathbb{N},$ let $\Upsilon(L)$ be the smallest constant such that the following inequality holds true:
\begin{align}
\label{Christ}
\big\|\sup_{k\in \mathbb{Z}: |k|\leq L} |S[\mu_k]f|\big\|_{L^p(\omega^s)}\leq \Upsilon(L)\, \big\|f\big\|_{L^p(\omega^s)}. 
\end{align}

The existence of such a finite $\Upsilon(L)$ is guaranteed by the subsequent observations. When $s=1,$ using our hypothesis \eqref{uniform} we obtain that 
\begin{align}\label{finite-1}\big\|   \sup_{k\in \mathbb{Z}: |k|\leq L} | S[\mu_k]f |  \big \|^{p}_{L^p(\omega)}\leq \sum_{|k|\leq L} \big \|S[\mu_k]f\|^{p}_{L^p(\omega)}\leq C L \, \big\|f\big\|_{L^p(\omega)}^{p}.\end{align}
Similarly, for $s=0$, simply have the fact that $S[\mu_k]$ maps $L^p(G)$ to itself with constant independent of $k$ for $1\leq p\leq \infty$ to conclude that
\begin{align}
\label{finite-2}
\big\|\sup_{k\in \mathbb{Z}: |k|\leq L} S[\mu_k]f\big \|^{p}_{L^p}\leq \sum_{|k|\leq L} \big \|S[\mu_k]f\|^{p}_{L^p}\leq C L \, \big\|f\big\|_{L^p}^{p}.
\end{align}
Stein--Weiss interpolation between \eqref{finite-1} and \eqref{finite-2} ensures the existence of $\Upsilon(L)$ in \eqref{Christ} for $0\leq s\leq 1.$

The rest of the proof is dedicated to conclude that $\Upsilon(L)\leq C$ uniformly in $L\in \mathbb{N}.$ We follow an argument of by M. Christ, see \cite{JV}. The uniform estimate \eqref{uniform} together with \eqref{Christ} implies the following vector-valued inequality
\begin{align}
\label{req-vector-valued-1}
\big\|\sup_{k\in \mathbb{Z}: |k|\leq L} S[\mu_k] (h_{k}\ast \phi_{k+j})\|_{L^p(\omega^s)}\leq \Upsilon(L)\, \big\|\sup_{k\in \mathbb{Z}: |k|\leq L} |h_{k}|\big\|_{L^p(\omega^s)}.  
\end{align}

To see the above inequality, we argue as follows. First, we observe
\begin{equation}\label{Maximal fn domination}
    \begin{split}
       ||h_{k}|  \ast \phi_{k+j} (x)| & \leq \int_{G} |h_{k}(x \, y^{-1})| | \phi_{k+j}(y) | dy \\
       & \lesssim 2^{-(k+j)Q} \int_{B(0, C 2^{(k+j)})} |h_{k}(x \, y^{-1})| dy \\
       & \overset{y \rightarrow y^{-1}x}{\lesssim} 2^{-(k+j)Q} \int_{B(x, C 2^{(k+j)})} |h_{k}(y)| dy \\
       & \leq \mathfrak{M}(|h_{k}|) (x).
    \end{split}
\end{equation}

Now, using \eqref{Christ} and positivity of $S[\mu_k]$,  we obtain
\begin{align*}
\big\|\sup_{k\in \mathbb{Z}: |k|\leq L}  \left|  S[\mu_k] (h_{k}\ast \phi_{k+j})  \right|   \|_{L^p(\omega^s)}&\lesssim \big\|\sup_{k\in \mathbb{Z}: |k|\leq L} S[\mu_k] \big( \mathfrak{M}(\sup_{|k|\leq L}  \left| h_{k}  \right| )\big)\|_{L^p(\omega^s)}\\
&\overset{\eqref{Christ}}{\lesssim} \Upsilon(L)\, \big\| \mathfrak{M}(\sup_{|k|\leq L}  \left| h_{k}  \right| )\|_{L^p(\omega^s)}\lesssim \Upsilon(L)\big\|\sup_{|k|\leq L}  \left| h_{k}  \right|\|_{L^p(\omega^s)},
\end{align*}
where the last inequality follows from the boundedness of $\mathfrak{M}$ on $L^p(\omega^s)$ to itself, and this in turn requires that $\omega^s\in \mathcal{A}_{p}(G),$ which is a simple consequence of $\omega\in \mathcal{A}_{p}(G)$ and H\"older's inequality. Next, uniform estimate \eqref{uniform} leads to one more vector-valued inequality
\begin{align*}
&\bigg\|\bigg(\sum_{k\in \mathbb{Z}: |k|\leq L} |S[\mu_k] (h_{k}\ast \phi_{k+j})|^{p}\bigg)^{\frac{1}{p}}\bigg\|_{L^p(\omega)}^{p}\\
& = \sum_{k\in \mathbb{Z}: |k|\leq L} \big\|S[\mu_k] (h_{k}\ast \phi_{k+j})\big\|_{L^p(\omega)}^{p}\\
&\overset{\eqref{uniform}}{\leq} C\, \sum_{k\in \mathbb{Z}: |k|\leq L} \big\|  h_{k}\ast \phi_{k+j} \big\|_{L^p(\omega)}^{p}   \overset{\eqref{Maximal fn domination}}{\lesssim} \bigg\|\bigg(\sum_{k\in \mathbb{Z}: |k|\leq L} |h_{k}|^{p}\bigg)^{\frac{1}{p}}\bigg\|_{L^p(\omega)}^{p}.    
\end{align*}

So, we obtain
\begin{align}
\label{req-vector-valued-2}
&\bigg\|\bigg(\sum_{k\in \mathbb{Z}: |k|\leq L} |S[\mu_k] (h_{k}\ast \phi_{k+j})|^{p}\bigg)^{\frac{1}{p}}\bigg\|_{L^p(\omega)}\lesssim C \bigg\|\bigg(\sum_{k\in \mathbb{Z}: |k|\leq L} |h_{k}|^{p}\bigg)^{\frac{1}{p}}\bigg\|_{L^p(\omega)}.
\end{align}
The interpolation of \eqref{req-vector-valued-1} and \eqref{req-vector-valued-2} implies that
\begin{align}
\label{req-vector-valued-3}
&\bigg\|\bigg(\sum_{k\in \mathbb{Z}: |k|\leq L} |S[\mu_k] (h_{k}\ast \phi_{k+j})|^{2}\bigg)^{\frac{1}{2}}\bigg\|_{L^p(\omega^r)}\lesssim C \Upsilon(L)^{1-\frac{p}{2}} \bigg\|\bigg(\sum_{k\in \mathbb{Z}: |k|\leq L} |h_{k}|^{2}\bigg)^{\frac{1}{2}}\bigg\|_{L^p(\omega^r)}.
\end{align}
for some $s<r<1$. At this stage, we wish to use the \eqref{req-vector-valued-3} with the functions $h_{k}=f\ast \psi_{j+k},$ where $\psi$ is a function, similar to $\phi,$ such that it satisfies $\psi\ast \phi=\phi,$ and then we  have
\begin{align}
\nonumber&\bigg\|\sup_{k\in \mathbb{Z}: |k|\leq L} |S[\mu_k] (f\ast \phi_{k+j})|\bigg\|_{L^p(\omega^r)}\leq \bigg\|\bigg(\sum_{k\in \mathbb{Z}: |k|\leq L} |S[\mu_k] (f\ast \psi_{k+j}\ast \phi_{k+j})|^{2}\bigg)^{\frac{1}{2}}\bigg\|_{L^p(\omega^r)}\\
&\overset{\eqref{req-vector-valued-3}}{ \lesssim } C \Upsilon(L)^{1-\frac{p}{2}} \bigg\|\bigg(\sum_{k\in \mathbb{Z}: |k|\leq L} |f\ast \psi_{k+j}|^{2}\bigg)^{\frac{1}{2}}\bigg\|_{L^p(\omega^r)}\lesssim C \Upsilon(L)^{1-\frac{p}{2}} \|f\|_{L^p(\omega^r)}, \label{bc}
\end{align}
where the last inequality follows from the  boundedness of square function on $L^{p}(\omega^{r})$, which in turn is consequence of $w \in \mathcal{A}_{p}(G)$, $0<r <1$ and Lemma~\ref{square-lemma}.

Recall that for $1<q< \infty$, one has the unweighted inequality
\begin{equation}\label{Lacunary-piece}
    \bigg\|\sup_{k\in \mathbb{Z}} |S[\mu_k] (f\ast \phi_{k+j})|\bigg\|_{L^q}\leq 2^{-|j|\delta(q)} \|f\|_{L^q}.
\end{equation}
For $1 < q \leq 2$, it follows immediately from Proposition 2.2 of \cite{HickmanJFA} after applying union bound argument. Then, interpolating this bound for small enough $q$ with the  corresponding $L^{\infty}$-bound where operator norm is $\lesssim 1$, yields \eqref{Lacunary-piece}.

Thus, interpolation between the estimate \eqref{bc} and the corresponding unweighted estimate \eqref{Lacunary-piece}, for the same $1 < p < 2$ implies that for any $0<s<1,$ there exists $\theta=\theta_{s}\in (0, 1)$, $\delta^{\prime}(p) >0$ such that
\begin{equation}
\label{req-almost}
\bigg\|\sup_{k\in \mathbb{Z}: |k|\leq L} |S[\mu_k] (f\ast \phi_{k+j})|\bigg\|_{L^p(\omega^s)}\leq 2^{-|j|\delta'(p)} \Upsilon(L)^{1-{\theta}} \|f\|_{L^p(\omega^s)}.   
\end{equation}
Now, summing the estimate \eqref{req-almost} in $j$, we obtain that  
\begin{equation}
\label{req-final}
\bigg\|\sup_{k\in \mathbb{Z}: |k|\leq L} |S[\mu_k] (f)|\bigg\|_{L^p(\omega^s)}\lesssim \Upsilon(L)^{1-{\theta}} \|f\|_{L^p(\omega^s)}.
\end{equation}
Since we assumed that $\Upsilon(L)$ is the smallest constant in \eqref{Christ}, the above estimate \eqref{req-final} implies that $\Upsilon(L)\lesssim \Upsilon(L)^{1-{\theta}},$ concluding that $\Upsilon(L)\lesssim 1$. This completes the proof in the case $1<p<2.$ 

The proof in the case $2\leq p<\infty$ is simpler in nature; however, we provide it for the sake of completeness. For any $\omega\in \mathcal{A}_{p}(G),$ observe that using a union bound argument and \eqref{uniform} we get
\begin{align}
\bigg\|\sup_{k\in \mathbb{Z}} |f\ast \phi_{k+j}\ast \mu_{k}|\bigg\|_{L^p(\omega)}^{p}
&\leq \bigg\| \bigg(\sum_{k\in \mathbb{Z}} |f\ast \phi_{k+j}\ast \mu_{k}|^{p} \bigg)^{\frac{1}{p}}\bigg\|_{L^p(\omega)}\nonumber\\
& = \sum_{k\in \mathbb{Z}} \int_{G} |f\ast \phi_{k+j}\ast \mu_{k}|^{p}\, \omega\, dx\nonumber\\
& \overset{\eqref{uniform}}{\lesssim} C \sum_{k\in \mathbb{Z}} \int_{G} |f\ast \phi_{k+j}|^{p}\, \omega\, dx \nonumber \\
& \leq C \int_{G} \left( \sum_{k\in \mathbb{Z}}|f\ast \phi_{k+j}|^{2}\right)^{\frac{p}{2}} \omega\, dx\ \ \ \text{(since}\, 2\leq p<\infty),\nonumber\\
&\lesssim C \int_{G} |f|^p\, \omega\,  dx \label{easy1},
\end{align}
where the last inequality follows from weighted boundedness of the square-function in Lemma~\ref{square-lemma}. Interpolating \eqref{easy1} and the unweighted estimate \eqref{Lacunary-piece} we obtain that
\begin{equation}\label{easy2}
    \bigg\|\sup_{k\in \mathbb{Z}} |S[\mu_k] (f\ast \phi_{k+j})|\bigg\|_{L^p(\omega^s)}\leq 2^{-|j|\delta_{s}(p)} \|f\|_{L^p(\omega^s)},
\end{equation}
for all $0\leq s<1$ for some $\delta_{s}(p)>0.$ Now summing \eqref{easy2} in $j$ the proof follows from \eqref{dyadic-piece}. This completes the proof for $2\leq p<\infty.$ 
\end{proof}

Now we prove the Theorem~\ref{Lac-end}.

\begin{proof}[Proof of Theorem~\ref{Lac-end}]
Since $\mu$ is compactly supported, we may assume that $supp(\mu)$ is contained in $B(0, \kappa)$ for some $\kappa>0.$ Let $R_{j}$ denotes the annulus $\{x\in G: 2^{j}<|x|\leq 2^{j+1}\}$ and we decompose the left hand side of \eqref{lac-end} as following
\begin{align*}
\int_{G}|S_{lac}[\mu]f(x)|^p |x|^{\alpha}\, dx&\leq \sum_{j\in \mathbb{Z}} \int_{R_{j}} \sup_{ k\in \Z}\big|S[\mu_k]f(x)\big|^p\,\, |x|^{\alpha}\, dx=\sum_{j\in \mathbb{Z}} \mathcal{I}_{j},
\end{align*}
where 
$$\mathcal{I}_{j}:=\int_{R_{j}} \sup_{ k\in \Z}\big|S[\mu_k]f(x)\big|^p\,\, |x|^{\alpha}\, dx.$$
We decompose the function $f$ as $f=g_{j}+h_{j}$, where $g_{j}=f \chi_{\{|x|\leq 2^{j+1}\}}$, and using this we
write
\begin{align}
\mathcal{I}_{j}\leq\int_{R_{j}} \sup_{ k\in \Z}\big|S[\mu_k]g_{j}(x)\big|^p\,\, |x|^{\alpha}\, dx+ \int_{R_{j}} \sup_{ k\in \Z}\big|S[\mu_k]h_{j}(x)\big|^p\,\, |x|^{\alpha}\, dx=\mathcal{I}_{j}^{1}+\mathcal{I}_{j}^2.  
\end{align}
For $\mathcal{I}_{j}^{1},$ using the unweighted boundedness of $S_{lac}[\mu]$, we obtain 
\begin{align}
\notag &\sum_{j\in \mathbb{Z}}\mathcal{I}_{j}^{1}\\
\notag= &\sum_{j\in \mathbb{Z}} \int_{R_{j}} \sup_{ k\in \Z}\big|S[\mu_k]g_{j}(x)\big|^p\,\, |x|^{\alpha}\, dx\\
\notag\lesssim  &\sum_{j\in \mathbb{Z}} 2^{j\alpha} \int_{R_{j}}  \sup_{ k\in \Z}\big|S[\mu_k]g_{j}(x)\big|^p\, dx\lesssim \sum_{j\in \mathbb{Z}} 2^{j\alpha} \int_{G}  \sup_{ k\in \Z}\big|S[\mu_k]g_{j}(x)\big|^p\, dx\\
\notag\lesssim  &\sum_{j\in \mathbb{Z}} 2^{j\alpha} \int_{|x|\leq 2^{j+1}}  |f(x)|^p\, dx\lesssim \sum_{j\in \mathbb{Z}} 2^{j\alpha} \sum_{\ell=-\infty}^{j}  \int_{2^{\ell}\leq |x|\leq 2^{\ell+1}}  |f(x)|^p\, dx\\
\notag &\lesssim \sum_{j\in \mathbb{Z}} 2^{j\alpha} \sum_{\ell=-\infty}^{j}  \int_{2^{\ell}\leq |x|\leq 2^{\ell+1}} 2^{-\ell \alpha}  |f(x)|^p\,|x|^{\alpha} dx\\
\notag &\lesssim \sum_{\ell=-\infty}^{\infty} 2^{-\ell \alpha} \sum_{j=\ell}^{\infty} 2^{j\alpha}   \int_{2^{\ell}\leq |x|\leq 2^{\ell+1}}  |f(x)|^p\,|x|^{\alpha}\,dx\\
\label{end-eq-1} &\overset{(\alpha<0)}{\lesssim} \sum_{\ell=-\infty}^{\infty}\int_{2^{\ell}\leq |x|\leq 2^{\ell+1}}  |f(x)|^p\,|x|^{\alpha}\,dx \lesssim \int_{G}  |f(x)|^p\,|x|^{\alpha}\,dx. 
\end{align}
To estimate $\mathcal{I}_{j}^2$ we make couple of observations. For $x\in R_{j}, y\in supp(\mu)$ and $k<j$ observe that $|x. \delta_{2^k} y^{-1}|\leq 2^{j+1}+2^{k}\lesssim
_{\kappa} 2^{j+1}$ as $k<j,$ hence $S[\mu_k]h_{j}(x)=0$ for $k<j.$ Also, for $x\in R_{j}, y\in supp(\mu)$ for a fixed $k>j$ we have $|x. \delta_{2^k} y^{-1}|\lesssim_{\kappa} 2^{j+1}+2^k\lesssim C_{\kappa} 2^{k}.$ Similarly, there exists a $c_{\kappa}$ such that $|x. \delta_{2^k} y^{-1}|\geq c_{\kappa} 2^{k}.$ Considering the above discussion we can write that $S[\mu_k](h_{j})(x)=S[\mu_k](f\chi_{\{c_{\kappa} 2^{k}\leq |x|\leq C_{\kappa}2^{k}\}})(x).$ Consequently,
\begin{align}
\notag
\notag &\sum_{j\in \mathbb{Z}}\mathcal{I}_{j}^{2}\\
\notag= &\sum_{j\in \mathbb{Z}} \int_{R_{j}} \sup_{ k\in \Z}\big|S[\mu_k]h_{j}(x)\big|^p\,\, |x|^{\alpha}\, dx\\
\notag\lesssim &\sum_{j\in \mathbb{Z}} 2^{j\alpha} \int_{R_{j}} \sum_{ k=j}^{\infty} \big|S[\mu_k](f\chi_{\{c_{\kappa} 2^{k}\leq |x|\leq C_{\kappa}2^{k}\}})(x)\big|^p\,dx\\
\notag \lesssim &\sum_{k\in \mathbb{Z}} 2^{j\alpha} \sum_{j=-\infty}^{k} \int_{R_{j}}\big|S[\mu_k](f\chi_{\{c_{\kappa} 2^{k}\leq |x|\leq C_{\kappa} 2^{k}\}})(x)\big|^p\,dx\\
\notag \lesssim &\sum_{k\in \mathbb{Z}}  \sum_{j=-\infty}^{k} \int_{R_{j}}\big|S[\mu_k](f\chi_{\{c_{\kappa} 2^{k}\leq |x|\leq C_{\kappa} 2^{k}\}})(x)\big|^p\, |x|^{\alpha}\, dx\\
\notag \lesssim &\sum_{k\in \mathbb{Z}} \int_{|x|\lesssim 2^{k}}\big|S[\mu_k](f\chi_{\{c_{\kappa} 2^{k}\leq |x|\leq C_{\kappa} 2^{k}\}})(x)\big|^p\, |x|^{\alpha}\, dx\\
 \overset{\eqref{unoform-end}}{\lesssim} &\sum_{k\in \mathbb{Z}} \int_{c_{\kappa} 2^{k}\leq |x|\leq C_{\kappa} 2^{k}}|f(x)|^p\,|x|^{\alpha}\,dx\lesssim \int_{G}|f(x)|^p\,|x|^{\alpha}\,dx.\label{end-eq2}
\end{align}
Combining \eqref{end-eq-1} and \eqref{end-eq2} we obtain that 
\begin{align*}
&\int_{G}|S_{lac}[\mu]f(x)|^p |x|^{\alpha}\, dx\\
&\leq\sum_{j\in \mathbb{Z}} \mathcal{I}_{j}\\
& \leq\sum_{j\in \mathbb{Z}} \mathcal{I}_{j}^{1}+ \sum_{j\in \mathbb{Z}} \mathcal{I}_{j}^{2}\\
&\lesssim \int_{G}|f(x)|^p\,|x|^{\alpha}\,dx.
\end{align*}
This completes the proof of \eqref{lac-end}.
\end{proof}

\section{Proof of Theorem~\ref{main:koranyi} and Theorem~\ref{main:hori}}
\label{Heisenberg-max}
The proof of Theorem~\ref{main:koranyi} and Theorem~\ref{main:hori} requires the following preliminary lemmas. We divide the section into two further subsections.

\subsection{Proof of Theorem~\ref{main:koranyi}}
The first step towards the proof is to understand the analogue of the $\mathcal{A}_{1}$ condition for  the Kor\'anyi spherical means. Recall, for $x:=(z,t) \in \mathbb{H}^n$,
\begin{align}
S_r f(x) =\int_{|(u, s)|=1} f\big(z-ru, t-r^2s-r\frac{1}{2}\Im(  z . \bar{u})\big) \, d\sigma(u, s). 
\end{align}
\begin{lemma}
\label{A1sph}
Let $w_{\alpha}=|x|^\alpha.$ Then $S_{r}(w_{\alpha})(x)\leq C_{\alpha, d}\, w_{\alpha}(x)\,\,\,a.e.$ if and only if
$$-(d-1)<\alpha\leq 0.$$
\end{lemma}
\begin{proof}
By scaling it is enough to consider $r=1$. We simply write $S_{1}$ by $S.$ We denote $x:=(z, t).$ The case when $|x|\gg 1,$ we can simply write 
\begin{align}
S(|y|^{\alpha})(x)= S(|y|^{\alpha})(z, t) &=\int_{|(u, s)|=1} \left|\Big(z-u, t-s-\frac{1}{2}\Im(  (z-u) . \bar{u})\Big) \right|^{\alpha} \, d\sigma(u, s)\nonumber\\
&\lesssim |x|^{\alpha},
\end{align}
since by the triangle inequality we must have $|(z, t). (u, s)^{-1}|\simeq |x|$ in this case. Also, when $|x|\ll 1,$ the proof follows considering the fact that $|\big(z-u, t-s-\frac{1}{2}\Im(  (z-u) . \bar{u})\big)|^{\alpha}\simeq 1\lesssim |x|^{\alpha}$ if and only if $\alpha\leq 0.$ The integrand may have singularities only when $|x|\simeq 1$ and in this case we follow an annular decomposition to conclude
\begin{align}
S(|y|^{\alpha})(x)= S(|y|^{\alpha})(z, t) &=\int_{|(u, s)|=1} |\big(z-u, t-s-\frac{1}{2}\Im(  (z-u) . \bar{u})\big)|^{\alpha} \, d\sigma(u, s)\nonumber\\
&\lesssim \sum_{k\geq 0}\int_{\Sigma_{K}\cap {A}^{x}_{k}} |\big(z-u, t-s-\frac{1}{2}\Im((z-u). \bar{u})\big)|^{\alpha} \, d\sigma(u, s)\nonumber,
\end{align}
where ${A}^{x}_{k}=\{y=(u, s)\in \mathbb{R}^{d}: \|x-y\|\simeq 2^{-k}\}$ for $k\geq 0$ . Observe that for $(u, s)\in {A}^{x}_{k}$ we must have
$$\|z-u\|+ |t-s-\frac{1}{2}\Im(z. \bar{u})|^{\frac{1}{2}}\geq 2^{-k},$$
which implies that
\begin{align}
S(|y|^{\alpha})(x)&\lesssim \sum_{k\geq 0}\int_{\Sigma_{K}\cap {A}^{x}_{k}} |\big(z-u, t-s-\frac{1}{2}\Im((z-u). \bar{u})\big)|^{\alpha} \, d\sigma(u, s)\nonumber\\
&\lesssim \sum_{k\geq 0} 2^{-k\alpha}\sigma(\Sigma_{K}\cap {A}^{x}_{k})\\
&\lesssim \sum_{k\geq 0} 2^{-k(\alpha+2n)}<\infty, 
\end{align}
only if $\alpha+2n>0,$ that is, $\alpha>-(d-1).$ Here, we have used the fact that $\sigma(\Sigma_{K}\cap {A}^{x}_{k}) \lesssim (2^{-k})^{2n}$. 
Near the north pole, $x=e_{2n+1}$, the Kor\'anyi sphere behaves as the surface $u_{2n+1}= 1 + \|\underline{u}\|^4$, $\|\underline{u}\| \ll 1$. So surface measure of the cap, there, is $\simeq r^{2n}$. Whereas, near the equator $x=e_{1}$ the Kor\'anyi sphere behaves like $u_{1}= 1 - {\frac{1}{2}} \sum_{j=2}^{2n} u_{j}^2 - \frac{1}{4} s^2$, therefore surface measure of the cap near equator is $\simeq r^{2n}$, since the Jacobian factors of the coordinate charts here is $\simeq 1$.

For necessity, we look $S(w_{\alpha})(x)$ near north pole, $e_{2n+1}=(\underline{0},1) \in \R^{2n} \times \R$. Let $0 < \delta \ll 1$. A simple  calculation reveals that
\begin{align}
&S(|y|^{\alpha})(e_{2n+1})\simeq \int_{0}^{{\delta}} \big(1-\sqrt{1-r^{4}})^{\alpha/4}\, r^{2n-1}\,\,dr.
\end{align}
The condition $\alpha>-2n=-(d-1)$ is necessary for the convergence of the above integral. Now we show that the same happens for a small a small neighbourhood of $e_{2n+1}$ given as
\begin{equation*}
    R_{\delta} :=  \left\lbrace (z,t) \in \R^{2n} \times \R : \| z\| \leq \delta, \, 0 \leq 1-t \leq \delta^4 \right\rbrace.
\end{equation*}
Given $x=(z,t) \in R_{\delta}$,
\begin{equation*}
    S(w_{\alpha})(x)= \int_{|(u, s)|=1} \left( \| z-u \|^{4} +  \left( t-s-\frac{1}{2}\Im(  z \cdot \bar{u}) \right)^2 \right)^{-|\alpha|/4} \, d\sigma(u, s),
\end{equation*}
 which integrating over a $\delta$-cap around $e_{2n+1}$ again and using local coordinates there, this integral is at least
 \begin{equation*}
     \gtrsim \int_{\| u \| \leq \delta} \left( \| z-u \|^{4} +  \left( t-  \sqrt{1 - \|u \|^4}     -\frac{1}{2}\Im(  z \cdot \bar{u}) \right)^2 \right)^{-|\alpha|/4} \, du.
 \end{equation*}
 For given $u$, Taylor expanding the integrand with respect to $(z,t)$ close to north pole, we can write the expression
 \[\| z-u \|^{4} +  \left( t- \sqrt{1 - \|u \|^4}    -\frac{1}{2}\Im(  z \cdot \bar{u}) \right)^2 \]
 as

 \begin{equation*}
     \begin{split}
         & \|u\|^4 + \left( 1 - \sqrt{1 - \|u \|^4} \right)^2 + 2 \left(1- \sqrt{1 - \|u \|^4}  \right) \, \langle(-u,1),(z, t-1) \rangle + O_{u,s} \left(\|z \| + (1-t) \right)\\
         & \simeq \|u\|^4 +  O \left( \|u\|^4 \left|\langle z, u \rangle +  1- t \right|  \right)\\
         & \simeq \|u\|^4 +  O \left( \delta \|u\|^4  \right)\\
         & \simeq \|u\|^4.
     \end{split}
 \end{equation*}
 Thus,
 \begin{equation*}
      S(w_{\alpha})(x) \gtrsim \int_{\| u \| \leq \delta}  \| u \| ^{-|\alpha|} \, du,
 \end{equation*}
 Hence, necessarily $\alpha > - 2n = -(d -1)$.
\end{proof}

We prove the following weighted estimate for every single averaging operator as an application of Lemma~\ref{A1sph} and a powerful factorization theorem for weights. 
\begin{theorem}[Power weights for a single averaging operator]
\label{single-Lp}
Let $r>0$. Recall that
$$ S_r f(x)=f\ast \sigma_r(x)= \int_{|y|=1} f(x \cdot \delta_r y^{-1}) \, d\sigma(y), \ \  x \in \Ha.$$
Let $1<p<\infty,$ then
$$\int_{\mathbb{H}^n} |S_{r}f(x)|^p |x|^{\alpha}\, dx\leq C_{n, \alpha} \int_{\mathbb{H}^n} |f(x)|^p |x|^{\alpha} dx,$$ 
holds for all $f\in L^p(\mathbb{H}^n, |x|^\alpha)$, provided
$$-(d-1)\leq \alpha\leq (d-1)(p-1).$$
\end{theorem}

\begin{proof}
The proof requires the very general factorization theorem:  $1\leq p<\infty$ and  $T$ be a positive linear self-adjoint operator, the weighted inequality $T:L^p(w)\to L^p(w)$ holds true if and only if there exists two weights $w_0, w_1$ such that $w = w_{0} \, w_{1}^{1-p}$ satisfying $Tw_{i}\lesssim w_{i}$ for $i=0, 1.$ See \cite{JW}, also \cite{JV}.

For $-(d-1)< \alpha<0,$ we simply take $w_{0}=|x|^{\alpha}, w_{1}=1,$ then the above factorization theorem together with Lemma~\ref{A1sph} concludes the proof. On the other hand, for $0<\alpha<(d-1)(p-1),$ one has to apply the above factorization and Lemma~\ref{A1sph} with $w_{0}=1, w_{1}=|x|^{-\frac{\alpha}{(p-1)}}.$ For the end-point $\alpha=-(d-1),$ we modify $w_0$ and $w_1$ as in  \cite{JV}. Let $0 < \gamma <1$. We take $$\omega_{0}=|x|^{-\gamma-d+1}\chi_{|x|<\frac{1}{4}}+|1-|x||^{-\gamma}\chi_{\frac{1}{4}<|x|<4}+|x|^{1-d}\chi_{|x|\geq 4},$$ and $$\omega_{1}=|x|^{-\gamma/(p-1)}\chi_{|x|<\frac{1}{4}}+|1-|x||^{-\frac{\gamma}{(p-1)}}\chi_{\frac{1}{4}<|x|<4}+\chi_{|x|\geq 4}.$$ All we need to prove is that $S(w_{i})\lesssim w_{i}$ for $i=0, 1$ and then the above factorization concludes the result. We provide a brief sketch of the estimate $S({\omega_{0}})\lesssim \omega_{0}.$ Observe that the first and third term in $\omega_{0}$ can be handled exactly as in the proof of Lemma~\ref{A1sph}.

While estimating the second term, that is $S(|1-|y||^{-\gamma}\chi_{\frac{1}{4}<|y|<4})(x),$ we need to handle singularities when the point $x=(z, t)$ satisfies $|x|\simeq 4.$ As in the proof of the sufficiency part in Lemma~\ref{A1sph}, we decompose the integral as
\begin{align}
\nonumber S(|1-|y||^{-\gamma}\chi_{\frac{1}{4}<|y|<4})(x)\lesssim \sum_{k\geq 0}\int_{\Sigma_{K}\cap \mathcal{R}^{(z, t)}_{k}} \big|1- |\big(z-u, t-s-\frac{1}{2}\Im((z-u). \bar{u})\big)|\big|^{-\gamma} \, d\sigma(u, s),
\end{align}
where $\mathcal{R}^{(z, t)}_{k}$ is the region such that either $1-\frac{c}{2^{k}}\lesssim \|(z, t)-(u, s)\|\lesssim 1-\frac{c}{2^{k+1}}$ or $1+\frac{c'}{2^{k+1}}\lesssim \|(z, t)-(u, s)\|\lesssim 1+\frac{c'}{2^{k}}.$ In the case when $1-\frac{c}{2^{k}}\lesssim \|(z, t)-(u, s)\|\lesssim 1-\frac{c}{2^{k+1}},$ one can prove that 
\begin{align}
\left| 1- \left( \|z-u\|^{4} + \big(t-s-\frac{1}{2}\Im (z-u.\bar{u})\big)^2 \right)^{1/4} \right| \gtrsim 2^{-k}.
\label{just-a-ineq}
\end{align} 
To see the above we simply use that $\big| \|z-u\|^4+(t-s)^2-(t-s) \Im (z-u. \bar{u})+\frac{1}{4}(\Im (z-u. \bar{u}))^2\big|\lesssim 1-2^{-k}$ for $1-\frac{c}{2^{k}}\lesssim \|(z, t)-(u, s)\|\lesssim 1-\frac{c}{2^{k+1}}.$ The estimate \eqref{just-a-ineq} implies $\big|1- |\big(z-u, t-s-\frac{1}{2}\Im((z-u). \bar{u})\big)|\big|^{-\gamma} \lesssim 2^{k \gamma}$ for $1-\frac{c}{2^{k}}\lesssim \|(z, t)-(u, s)\|\lesssim 1-\frac{c}{2^{k+1}}.$ A similar argument holds for the case $1+\frac{c'}{2^{k+1}}\lesssim \|(z, t)-(u, s)\|\lesssim 1+\frac{c'}{2^{k}}.$ Finally, considering the fact that $\sigma(\Sigma_{K}\cap \mathcal{R}^{z, t}_{k})\simeq (1-\frac{c}{2^{k+1}})^{2n}-(1-\frac{c}{2^{k}})^{2n}\simeq \frac{1}{2^{k}},$ we obtain that
\begin{align}
S(|1-|y||^{-\gamma}\chi_{\frac{1}{4}<|y|<4})(x)&\lesssim \sum_{k\geq 0}\int_{\Sigma_{K}\cap \mathcal{R}^{(z, t)}_{k}} \big|1- |\big(z-u, t-s-\frac{1}{2}\Im((z-u). \bar{u})\big)|\big|^{-\gamma} \, d\sigma(u, s)\nonumber\\
\nonumber &\lesssim \sum_{k\geq 0} 2^{k\gamma} 2^{-k}\lesssim 1\lesssim \omega_{0}(x), 
\end{align}
for $|x|\simeq 4.$ The proof of $S(\omega_{1})\lesssim \omega_{1}$ follows similar line of argument. The weight $|x|^{(d-1)(p-1)}$ is obtained by duality. 
\end{proof}  

\begin{proof}[Proof of Theorem~\ref{main:koranyi}] 
First, let us recall from Lemma~3.2 in \cite{HickmanJFA} that the surface measure of the Kor\'anyi sphere satisfies the \textit{Curvature assumption (CA).} Therefore, our Theorem~\ref{main:suff} is applicable for the lacunary maximal function $\mathcal{M}_{lac}.$

Let $1<p<\infty.$ Assume that $-(d-1)\leq \alpha< (d-1)(p-1).$ Observe that for each $-(d-1)\leq \beta\leq (d-1)(p-1),$ Theorem~\ref{single-Lp} concludes that each $k\in \mathbb{Z},$ the averaging operator $S_{2^{k}}=f\ast \sigma_{k}$ is bounded from $L^p(\Ha, |x|^\beta)$ to itself with constant independent of $k.$ Moreover, for this range of $\beta,$ $|x|^\beta\in \mathcal{A}_{p}(\Ha).$ Now a straight-forward application of Theorem~\ref{main:suff} concludes that $\mathcal{M}_{lac}$ is bounded from $L^p(\Ha, |x|^{\beta s})$ to itself for all $0\leq s<1.$ This in turn implies that $\mathcal{M}_{lac}$ is bounded from $L^p(\Ha, |x|^{\alpha})$ to itself for $-(d-1)<\alpha< (d-1)(p-1).$ For the end-point $\alpha=-(d-1),$ the proof follows from Theorem~\ref{single-Lp} and Theorem~\ref{Lac-end}.   
\end{proof}

\subsection{Proof of Theorem~\ref{main:hori}}

Recall
\begin{equation}\label{eq: horizontal spherical av}
    H_{r} f(z, t) = \int_{\mathbb{S}^{2n-1}} f\big(z-r u, t- {\textstyle\frac{r}{2} }\Im(z. \bar{u})\big) \,d \sigma_{2n-1}(u), \qquad (z, t) \in \Ha.
\end{equation}

We start with our first result describing $\mathcal{A}_{1}-type$ condition in the context of the context of the means $H_{r}.$  
\begin{lemma}
Let $w_{\alpha}=|x|^\alpha.$ Then $H_{r}(w_{\alpha})(x)\leq C_{\alpha, d}\, w_{\alpha}(x)\,\,\,a.e.$ if and only if
$$-(2n-1)<\alpha\leq 0.$$
\end{lemma}

\begin{proof}
The proof follows ideas from Lemma~\ref{A1sph}. By scaling it is enough to consider $r=1$. We simply write $H_{1}$ by $H.$ We denote $x:=(z, t).$ We begin with showing the sufficiency of the condition $-(2n-1) < \alpha \leq 0$.

The case when $|x|\gg 1$, then   $|x \cdot (u,0)^{-1} | \simeq |x|$. Whence,
\begin{align}
H(|y|^{\alpha})(x)= H(|y|^{\alpha})(z, t) &=    \int_{|u|=1} |x \cdot (u,0)^{-1} |^{\alpha} d\sigma_{2n-1}(u) \lesssim |x|^{\alpha},
\end{align}
 Also, when $|x| \ll 1$ then $  |x \cdot (u,0)^{-1} | \simeq 1   $ which leads to $  |x \cdot (u,0)^{-1} |^{\alpha} \simeq 1 \lesssim |x|^{\alpha}  $, because $\alpha\leq 0$. Thus, again we have $H(|y|^{\alpha})(x) \lesssim |x|^{\alpha}$.
 
The region, $|x| \simeq 1$, containing the singularity of the integral in (\ref{eq: horizontal spherical av}) on $\mathbb{S}^{n-1}$, is tackled by decomposing away from $\mathbb{S}^{n-1}$: 
\begin{align}
H(|y|^{\alpha})(x)= H(|y|^{\alpha})(z, t) &=\int_{|u|=1} |\big(z-u, t-\frac{1}{2}\Im(  (z-u) . \bar{u})\big)|^{\alpha} \, d\sigma_{2n-1}(u)\nonumber\\
&\lesssim \sum_{k\geq 0}\int_{\mathbb{S}^{2n-1}\cap B^{z}_{k}} |\big(z-u, t-\frac{1}{2}\Im(  (z-u) . \bar{u})\big)|^{\alpha} \, d\sigma_{2n-1}(u)\nonumber,
\end{align}
where $B^{z}_{k}=\{u\in \mathbb{R}^{2n}: \|z-u\|\simeq 2^{-k}\}$ for $k\geq 0.$ For $u\in B^{z}_{k}$ we must have
\begin{equation*}
   \|z-u\|+ |t-\frac{1}{2}\Im(z. \bar{u})|^{\frac{1}{2}}\geq 2^{-k}, 
\end{equation*}
which implies that
\begin{align}
H(|y|^{\alpha})(x)&\lesssim \sum_{k\geq 0}\int_{\mathbb{S}^{2n-1}\cap B^{z}_{k}} |\big(z-u, t-\frac{1}{2}\Im((z-u). \bar{u})\big)|^{\alpha} \, d\sigma_{2n-1}(u)\nonumber\\
&\lesssim \sum_{k\geq 0} 2^{-k\alpha}\sigma_{2n-1}(\mathbb{S}^{2n-1}\cap B^{z}_{k})\\
&\lesssim \sum_{k\geq 0} 2^{-k(\alpha+2n-1)} \lesssim 1, 
\end{align}
since $\alpha > -(2n - 1)$.  

Next, we prove the necessary part. The necessity of $\alpha \leq 0$ follows from looking at $H(w_{\alpha})(x)$, when $|x | \ll 1$. Indeed, if $\alpha$ were positive, then for $|x| \ll 1$,
\begin{equation*}
    \begin{split}
       |x|^{\alpha} \gtrsim H w_{\alpha}(x) & = \int_{\mathbb{S}^{2n-1}} \left| x \cdot (u,0)^{-1} \right|^{\alpha} d \sigma_{2n-1}(u)\\
        & \gtrsim  \int_{\mathbb{S}^{2n-1}} \left( 1 - |x| \right)^{\alpha} d \sigma_{2n-1}(u)\\
        & \gtrsim  \int_{\mathbb{S}^{2n-1}} 1 \, d \sigma_{2n-1}(u) \simeq 1,
    \end{split}
\end{equation*}
contradicting $|x| \ll 1$.

Let $x=(z, t) =(e_{1},0, 0) \in \mathbb{R}^{n} \times \mathbb{R}^{n} \times \mathbb{R} $  be the point on  the equator of $\Sigma_{K}$. Then,
\begin{align}
&H(|y|^{\alpha})(x)=\int_{ \|u\|=1}\big( \|e_1-u\|^{4}+\frac{1}{4}|\Im(e_1 \cdot \bar{u})|^2 \big)^{\alpha/4}\, d\sigma_{2n-1}(u)\\
\nonumber&\simeq \int_{\|u\|=1}\left(\big((1- u_1)^{2}+u_{2}^2+\cdots+u_{2n}^2\big)^2+
\frac{1}{4}u_{n+1}^2 \big)\right)^{\alpha/4}\, d\sigma_{2n-1}(u)\\
\nonumber&\simeq \int_{\|u\|=1}\left( 4(1- u_1)^{2} +
\frac{1}{4}u_{n+1}^2 \big)\right)^{\alpha/4}\, d\sigma_{2n-1}(u)\\
\notag &\gtrsim \int_{ (u_2, \dots, u_{2n}) \in B^{2n-1}(0, \delta)  }\left( \bigg( \sum_{j=2}^{2n} u_{j}^2 \bigg)^2 + 
\frac{1}{4}u_{n+1}^2 \right)^{\alpha/4}\, du_2 \cdots du_{2n}\\
\notag & \gtrsim \int_{ (u_2, \dots, u_{2n}) \in B^{2n-1}(0, \delta)   } \| (u_2, \dots, u_{2n})  \|^{\alpha}\, du_2 \cdots du_{2n},
\end{align}
for some ball $B^{2n-1}(0, \delta)$ of small radius $\delta$.
Therefore, the condition $\alpha>-(2n-1)$ is necessary for the convergence of the above integral.
\end{proof}

\begin{theorem}[Power weights for a single averaging operator]
\label{single-hori-Lp}
Let $r>0.$ Let $1<p<\infty,$ then
$$\int_{\mathbb{H}^n} |H_{r}f(x)|^p |x|^{\alpha}\, dx\leq C_{n, \alpha} \int_{\mathbb{H}^n} |f(x)|^p |x|^{\alpha} dx,$$ 
holds for all $f\in L^p(\mathbb{H}^n, |x|^\alpha)$ provided
$$-(2n-1)\leq \alpha\leq (2n-1)(p-1).$$
\end{theorem}
\begin{proof}
The proof follows from the arguments in Theorem~\ref{single-Lp} with necessary modifications.       
\end{proof}

\begin{proof}[Proof of Theorem~\ref{main:hori}] We recall from Lemma~3.3 in \cite{HickmanJFA} that the surface measure of the horizontal sphere in $\Ha$ satisfies the \textit{Curvature assumption (CA).} Therefore, our Theorem~\ref{main:suff} is applicable for the lacunary maximal function $M_{lac}.$

Let $1<p<\infty.$ Assume that $-(2n-1)\leq \alpha< (2n-1)(p-1).$ Observe that for each $-(2n-1)\leq \beta\leq (2n-1)(p-1),$ Theorem~\ref{single-hori-Lp} concludes that each $k\in \mathbb{Z},$ the averaging operator $H_{2^{k}}$ is bounded from $L^p(\Ha, |x|^\beta)$ to itself with constant independent of $k.$ Moreover, for this range of $\beta,$ $|x|^\beta\in \mathcal{A}_{p}(\Ha).$ Now a straight-forward application of Theorem~\ref{main:suff} concludes that ${M}_{lac}$ is bounded from $L^p(\Ha, |x|^{\beta s})$ to itself for all $0\leq s<1.$ This in turn implies that ${M}_{lac}$ is bounded from $L^p(\Ha, |x|^{\alpha})$ to itself for $-(2n-1)<\alpha< (2n-1)(p-1).$ For the end-point $\alpha=-(2n-1),$ the proof follows from Theorem~\ref{single-hori-Lp} and Theorem~\ref{Lac-end}.   
\end{proof}

\newcommand{\etalchar}[1]{$^{#1}$}
\providecommand{\bysame}{\leavevmode\hbox to3em{\hrulefill}\thinspace}
\providecommand{\MR}{\relax\ifhmode\unskip\space\fi MR }
\providecommand{\MRhref}[2]{%
  \href{http://www.ams.org/mathscinet-getitem?mr=#1}{#2}
}
\providecommand{\href}[2]{#2}

\end{document}